\documentclass[draft,12pt, reqno]{amsart}
\usepackage{amssymb, amsmath}
\usepackage{url}
\usepackage{enumerate}
\usepackage{mathabx}

\newtheorem{theorem}{Theorem}[section]
\newtheorem{corollary}[theorem]{Corollary}

\newtheorem{lemma}[theorem]{Lemma}

\theoremstyle{remark}

\theoremstyle{definition}
\newtheorem{defn}[theorem]{Definition}
\newtheorem{example}[theorem]{Example}

\numberwithin{equation}{section}
\numberwithin{theorem}{section}

\newcommand{\N}{{\mathbb N}}
\newcommand{\R}{{\mathbb R}}

\newcommand{\fn}{\!:\!}

\newcommand{\fhat}{\hat{f}}

\providecommand{\abs}[1]{\lvert#1\rvert}
\providecommand{\norm}[1]{\lVert#1\rVert}

\newcommand{\intinf}{\int_{-\infty}^{\infty}}

\begin{document}
\date{Preprint August 10, 2018.  To appear in {\it American Mathematical Monthly}.}
\title[Fourier transform inversion]
{Fourier transform inversion using an elementary differential equation and a contour integral}
\author{Erik Talvila}
\address{Department of Mathematics \& Statistics\\
University of the Fraser Valley\\
Abbotsford, BC Canada V2S 7M8}
\email{Erik.Talvila@ufv.ca}

\begin{abstract}
Let $f$ be a function on the real line.  
The Fourier transform inversion theorem
is proved under the assumption
that $f$ is absolutely continuous such that
$f$ and $f'$ are Lebesgue integrable.  A function $g$ is defined by $f'(t)-iwf(t)=g(t)$.
This 
differential equation has a well known integral solution using the
Heaviside step function.  An elementary calculation with residues is used to write the Heaviside
step function as a
simple contour integral.
The rest of the proof requires elementary
manipulation of integrals.  Hence, the Fourier transform inversion theorem is proved with very little
machinery.  With only minor changes the method is also used to prove the inversion theorem for functions
of several variables and to prove Riemann's localization theorem.
\end{abstract}

\maketitle

\section{Introduction}\label{sectionintroduction}
If you can solve the differential equation $f'(t)-iwf(t)=g(t)$ and compute a simple contour integral
then you can prove the Fourier transform inversion theorem.

In this paper we prove an elementary inversion formula for the Fourier transform.
Our proof depends on a contour integral representation
of the Heaviside step function and requires
solving a linear first order ordinary differential
equation, for which a solution integral is well known.  
We have stated our results for Lebesgue integrals but it is easy enough to rephrase them for
absolutely convergent
improper Riemann integrals.

If $f\fn\R\to\R$ then its Fourier transform is $\fhat(s)=\intinf e^{-isx}f(x)\,dx$ where $s\in\R$.
A sufficient condition for existence of $\fhat$ on $\R$ is that $\intinf \abs{f(x)}\,dx$ exists,
i.e., $f\in L^1(\R)$.  In this case, $\fhat$ is uniformly continuous on $\R$ and $\lim_{\abs{s}\to\infty}
\fhat(s)=0$ (Riemann--Lebesgue lemma).  The inversion formula is
\begin{equation}
f(x)=\frac{1}{2\pi}\intinf e^{isx}\fhat(s)\,ds.\label{inversionformula}
\end{equation}
However, this equation must be taken with a grain of salt since in general the integral diverges!  
If $f\in L^1(\R)$ then $\fhat$ is continuous so we will have local integrability of $e^{isx}\fhat(s)$
over compact intervals in $s$ but
unless further conditions are imposed on $f$ then we need not have  $\fhat\in L^1(\R)$.

In Theorem~\ref{theoremtransform} we prove an inversion theorem 
under the hypothesis that $f$ is
absolutely continuous with $f$ and $f'$ in
$L^1(\R)$
and then the integral in \eqref{inversionformula} converges as a principal value integral in the 
sense $\lim_{R\to\infty}\int_{-R}^R  e^{isx}\fhat(s)\,ds$.  

Our proof readily extends to functions of several variables, for which we need a discussion of
Carath\'{e}odory's notion of absolute continuity.  We obtain the inversion theorem
in $\R^2$ in Theorem~\ref{theoremtransformxy}, from which it is clear how inversion in $\R^n$ would be proved.  
The method also makes it easy to prove Riemann's localization 
result (Corollary~\ref{corollarylocalisation}).

The literature on Fourier analysis is vast and there are many methods of obtaining formulas similar to \eqref{inversionformula}
under various hypotheses.  One such hypothesis is that if $f$ and $\fhat$ are in $L^1(\R)$ then \eqref{inversionformula}
holds \cite[Theorem~8.26]{folland}.  A sufficient condition for $\fhat\in L^1(\R)$ is that $f,f',f''\in L^1(\R)$ with
$f'$ absolutely continuous.  The result follows integration by parts using Lemma~\ref{lemmaff'} below-
see also \cite{gasquet}.
Some other methods are: extending Fourier series results \cite{bachman}, \cite{boas};
convolution with a summability kernel (approximate identity) in the $L^p$ norm for 
$1\leq p\leq 2$ \cite{bachman}, \cite{folland} (When $p=2$ this is called convergence in the mean sense.);
spectral resolution of the differential operator $d^2/dx^2 + s^2$ using the Titchmarsh--Weyl--Kodaira theorem and
integral representation of the Green function \cite[pp.~1380-1382]{dunfordschwartz}, \cite[7.3]{stakgold},
\cite{titchmarsheigen}, \cite[p.~256]{weidmann}, \cite[p.~194]{yosida}.
Various other methods are covered in \cite{bachman} and \cite{titchmarsh}.

Many of the methods in the above literature require substantial machinery.
The proofs we give do not depend on any other results from
Fourier analysis but only on elementary manipulations of integrals on the real line and in the
complex plane, and solution of a simple ordinary differential equation.  We will see that there are two essential
ingredients.  One is the fundamental theorem of calculus in the form $\int_a^bF'(x)\,dx=F(b)-F(a)$
for an absolutely continuous function $F$ and $d/dx \int_{-\infty}^x f(t)\,dt=f(x)$ almost everywhere for an
$L^1$ function $f$.  Notice that absolute continuity is a local condition and tells us that
$\abs{F'}$ is integrable on compact intervals but says nothing about integrability of $F'$ over $\R$.
Our inversion theorem (Theorem~\ref{theoremtransform}) includes an hypothesis that gives integrability of
the derivative.  The other ingredient in our proof is the Fourier transform of the function $t\mapsto 1/(t-w)$ for a complex parameter
$w$ with nonzero imaginary part.   A contour integral calculation shows this gives  the product of a complex exponential
with the
Heaviside step function.
Hence, once the Fourier integral representation is
known in this one case, it follows for all absolutely continuous functions.    

If the Dirac distribution can be written as a Fourier transform this gives a distributional
(generalized function) version of the inversion theorem.  In some sense our method is an
integral version of this since the Heaviside step function is the integral of the Dirac distribution.
However, our inversion formula holds pointwise everywhere rather than in a distributional sense.

\section{Fourier transform inversion}\label{sectionft}
The key to our proof of the inversion theorem is the following integral representation of the Heaviside step
function.  The Heaviside step function is
\begin{equation}
H(x)=\left\{\begin{array}{cl}
1, \text{ if } x>0\\
1/2, \text{ if } x=0\\
0, \text{ if } x<0.
\end{array}
\right.
\end{equation}
\begin{lemma}\label{lemmaperron1}
Let $p$ be a real number and let $w$ be a complex number with positive imaginary part.  Then
$$
\frac{1}{2\pi i}\intinf \frac{e^{ipz}}{z-w}dz =e^{ipw}H(p).
$$
Carrying out the same calculation with $w=0$ shows that $\int_0^\infty [\sin(x)/x]\,dx=\pi/2$.
\end{lemma}
If $p\not=0$,
the integral converges conditionally as an improper Riemann integral (and as a Henstock--Kurzweil
integral and as a Cauchy--Lebesgue integral).
If $p=0$ or $w$ is real, the integral exists in the principal value sense.  Sometimes this is known as Perron's lemma.
For a proof see Section~\ref{sectionperron}.
Although it is an old result we include a proof because we will need estimates from it to justify
some limit operations in the proof of Theorem~\ref{theoremtransform}.
The lemma was used by
Perron for finding partial sums of Dirichlet series; for example, \cite[\S11.12]{apostol}.

Now we are ready for the inversion theorem.
\begin{theorem}\label{theoremtransform}
Let $f$ be an absolutely continuous function on $\R$ such that $f$ and $f'$ are in $L^1(\R)$.  For each $x\in\R$,
$$
f(x)=\frac{1}{2\pi}\lim_{R\to\infty}\int_{-R}^R e^{ixs}\intinf e^{-ist}f(t)\,dt\,ds.
$$
\end{theorem}
In the proof of the theorem $f$ is written as the solution of a simple differential equation.  Note, however,
that the inversion formula can hold at points where $f$ is not differentiable, provided $f$ is absolutely
continuous.  At points of differentiability the method of P.R. Chernoff can also be applied to give a proof
\cite{koekoek}.
\begin{proof}
Write $w=\xi+i\eta$ with $\xi\in\R$ and $\eta>0$.
Define $g=f'-iw f$.  This determines $g$ almost everywhere.  Then
$$
f(x)=e^{iwx}\int_{-\infty}^x e^{-iwt}g(t)\,dt.
$$
The condition $\eta>0$ ensures this integral exists for each $x\in\R$.
The fundamental theorem of calculus for Lebesgue integrals shows that $f$ satisfies the differential
equation for almost all $x$.  To show the differential equation has a unique absolutely continuous solution, suppose there were solutions
$u_1$ and $u_2$.  Let $u=u_1-u_2$.  Then $u'(x)-iw\,u(x)=0=e^{iw x}\frac{d}{dx}\left[e^{-iw x}u(x)\right]$.
Notice that $e^{iw x}=e^{-\eta x}[\cos(\xi x)+i\sin(\xi x)]$, which does not vanish for any $x$.
Since $u$ is absolutely continuous we then have $u(x)=ce^{iw x}$ for some constant $c$.
But $u\in L^1(\R)$ so $c=0$.

We can now use the integral representation of the Heaviside step function in Lemma~\ref{lemmaperron1} to write
\begin{equation}
f(x)  =  \intinf H(x-t)e^{iw (x-t)}g(t)\,dt
  =  \frac{1}{2\pi i}\intinf\lim_{R\to\infty}\int_{-R}^R\frac{e^{i(x-t)z}}{z-w}\,dz\,g(t)\,dt.\label{Rlimitint}
\end{equation}
We would like to bring the limit as $R\to\infty$ outside the integral.  Since $f,f'\in L^1(\R)$ we have
$g\in L^1(\R)$ and it suffices
to show the integral over $z$ is bounded as a function of $t$, uniformly as $R\to\infty$.  The proof of 
Lemma~\ref{lemmaperron1} shows this.  In that proof (below), the required integral is the difference of the residue
and the integral $I_R$.  When the imaginary part of $w$ is positive, the residue, $e^{i(x-t)w}$, has modulus at most
$1$.  The estimate in equation \eqref{contourestimate} shows $I_R$ is bounded as $R\to\infty$.
Then, by dominated convergence,
\begin{equation}
f(x)
  =  \frac{1}{2\pi i}\lim_{R\to\infty}\intinf\int_{-R}^R\frac{e^{i(x-t)z}}{z-w}\,dz\,g(t)\,dt.
\end{equation}
The estimate $\abs{e^{i(x-t)z}/(z-w)}\leq 1/\eta$ shows we can use the Fubini--Tonelli theorem
to interchange the $z$ and $t$ integrals.  Then integrate by parts.  We have
\begin{eqnarray*}
f(x)
 & = & \frac{1}{2\pi i}\lim_{R\to\infty}\int_{-R}^R\intinf\frac{e^{i(x-t)z}}{z-w}[f'(t)-iw f(t)]\,dt\,dz\\
 & = &  \frac{1}{2\pi i}\lim_{R\to\infty}\int_{-R}^R\frac{e^{ixz}}{z-w}\left\{\left[e^{-itz}f(t)\right]_{-\infty}^\infty
+i(z-w)\intinf e^{-itz}f(t)\,dt\right\}\,dz\\
 & = & \frac{1}{2\pi}\lim_{R\to\infty}\int_{-R}^R e^{ixz}\intinf e^{-itz}f(t)\,dt\, dz,
\end{eqnarray*}
provided $\lim_{\abs{x}\to\infty}f(x)=0$.  But this is shown in Lemma~\ref{lemmaff'} below.

\end{proof}
\begin{lemma}\label{lemmaff'}
Suppose $f,f'\in L^1(\R)$ and $f$ is absolutely continuous.  Then\\
$\lim_{\abs{x}\to\infty}
f(x)=0$.
\end{lemma}
\begin{proof}
It suffices to prove the lemma for $x\to\infty$.  Let $x>0$.  Integration by parts establishes
that
$$
f(x)=\frac{1}{x}\int_0^xf(t)\,dt+\frac{1}{x}\int_0^xtf'(t)\,dt.
$$
And,
$$
\abs{f(x)}\leq \frac{\norm{f}_1}{x}+\norm{f'}_1.
$$
Since $f$ is continuous, this shows that $f$ is bounded on $\R$.

Similarly, let $x>1$.  Then
$$
f(x)=\frac{1}{x}\int_{\sqrt{x}}^xf(t)\,dt+\frac{1}{x}\int_{\sqrt{x}}^xtf'(t)\,dt
+\frac{f(\sqrt{x})}{\sqrt{x}}.
$$
And,
$$
\abs{f(x)}\leq \frac{\norm{f}_1}{x}+\int_{\sqrt{x}}^x\abs{f'(t)}\,dt +\frac{\norm{f}_\infty}{\sqrt{x}}.
$$
\end{proof}

If $f\in L^1(\R)$ then $\fhat$ need not be in $L^1(\R)$.  For example, let $f(x)=1$ if $\abs{x}\leq 1$
and $f(x) =0$, otherwise.  Then $\fhat(s)=2\sin(s)/s$ for $s\not=0$ and $\fhat(0)=2$.  Although $\fhat$
is continuous and has limit $0$ at infinity it is not in $L^1(\R)$.  However, it has a conditionally
convergent integral. It is not known if the conditions of Theorem~\ref{theoremtransform} imply
$\fhat\in L^1(\R)$.

It is known from an example due to du Bois-Reymond (1876) that there is a continuous function whose
Fourier series diverges at a point.  The same applies to Fourier transforms.  In the following example,
based on \cite[\S4.12]{hardy},
we construct a continuous function in $L^1(\R)$ for which the inversion formula of Theorem~\ref{theoremtransform}
fails to hold.
\begin{example}\label{hardyexample}
By the Fubini--Tonelli theorem we can write the inversion formula as
$$
f(x)=\frac{1}{\pi}\lim_{R\to\infty}\intinf f(t+x)\frac{\sin(Rt)}{t}\,dt.
$$
Define $f\fn\R\to\R$ by
$$
f(t)=\left\{\begin{array}{cl}
a_k\sin(n_kt), & \pi/n_k\leq t\leq \pi/n_{k-1}\quad\text{ for some } k\in\N\\
0, & \text{ otherwise.}
\end{array}
\right.
$$
The sequences $a_k=1/k^2$ and $N_k=2^{k^3}$ are such that $\sum a_k<\infty$ and $a_k\log(N_k)\to\infty$ as
$k\to\infty$.  Define $n_k=N_1N_2\cdots N_k$ for $k\in\N$ and $n_0=1$.  Then $f$ is continuous and vanishes
outside the interval $[0,\pi]$ so $f\in L^1(\R)$.  Note that $f$ is not absolutely continuous.  For,
\begin{eqnarray*}
\int_0^\pi\abs{f'(t)}\,dt & = & \sum_{k=1}^\infty a_k n_k\int_{\pi/n_k}^{\pi/n_{k-1}}\abs{\cos(n_kt)}\,dt\\
 & = & \sum_{k=1}^\infty a_k\int_{\pi}^{N_k\pi}\abs{\cos(t)}\,dt=2\sum_{k=1}^\infty a_k(N_k-1)=\infty.
\end{eqnarray*}
This shows $f$ is not of bounded variation and so not absolutely continuous.  With this function $f$ the
hypotheses of Theorem~\ref{theoremtransform} are not fulfilled.

Note that $f(0)=0$.  We will show the inversion theorem fails at the origin.  Since $n_k\to\infty$,
it suffices to let $J_k=\int_0^\pi f(t)(\sin(n_kt)/t)\,dt$ and show $J_k$ does not have limit $0$ as $k\to\infty$.
We have
\begin{equation}
J_k=\sum_{m=1}^\infty a_m\int_{\pi/n_m}^{\pi/n_{m-1}}\frac{\sin(n_kt)\sin(n_mt)}{t}\,dt.\label{Jk}
\end{equation}
In this series the term with $m=k$ is
\begin{eqnarray*}
\frac{a_k}{2}\int_{\pi/n_k}^{\pi/n_{k-1}}\frac{1-\cos(2n_kt)}{t}\,dt & = & 
\frac{a_k}{2}\log(n_k/n_{k-1})-\frac{a_k}{2}\int_{2\pi}^{2\pi n_k/n_{k-1}}\frac{\cos(t)}{t}\,dt\\
 & = & \frac{a_k}{2}\log(N_k)-\frac{a_k}{2}\int_{2\pi}^{2\pi N_k}\frac{\cos(t)}{t}\,dt.
\end{eqnarray*}
Integration by parts establishes that if $a>0$ is fixed then $\int_a^y(\cos(t)/t)\,dt$ is bounded for $y>a$.
(For $y=\infty$ the integral converges conditionally.)  
We can then define $C=\sup_{\pi/2\leq x<y}\abs{\int_x^y(\cos(t)/t)\,dt}$
and
$$
\frac{a_k}{2}\int_{\pi/n_k}^{\pi/n_{k-1}}\frac{1-\cos(2n_kt)}{t}\,dt\geq \frac{a_k}{2}\log(N_k)-
\frac{a_k}{2}C.
$$
Since $a_k\to 0$ and $a_k\log(N_k)\to\infty$ this term of the series tends to infinity.  

Now show the
sum over all other terms remains finite as $k\to\infty$.  For $m\not= k$ the summands of \eqref{Jk}
are
\begin{align*}
&\left|\frac{a_m}{2}\int_{\pi/n_m}^{\pi/n_{m-1}}\frac{\cos(\abs{n_k-n_m}t)-\cos((n_k+n_m)t)}{t}\,dt\right|\\
&=\left|\frac{a_m}{2}\int_{\pi\abs{n_k-n_m}/n_m}^{\pi\abs{n_k-n_m}/n_{m-1}}\frac{\cos(t)}{t}\,dt
-\frac{a_m}{2}\int_{\pi(n_k+n_m)/n_m}^{\pi(n_k+n_m)/n_{m-1}}\frac{\cos(t)}{t}\,dt\right|\\
&\leq a_mC.
\end{align*}
We see this because $\abs{n_k\pm n_m}/n_m=\abs{n_k/n_m\pm1}$.  If $k>m$ then $n_k/n_m\geq N_1=2$.
If $k<m$ then $n_k/n_m\leq 1/N_1=1/2$.
\end{example}

Riemann's localization theorem was originally given for Fourier series but applies equally
well to Fourier transforms.  It says that the inversion formula for $f(x)$ in Theorem~\ref{theoremtransform}
only depends on the behavior of $f$ in a neighborhood of $x$.  We get this as a corollary,
following the method of proof in Theorem~\ref{theoremtransform}.
\begin{corollary}[Localization]\label{corollarylocalisation}
Let $f$ be an absolutely continuous function on the interval $[x_1,x_2]$.
Let $x_1<x<x_2$. Then
$$
f(x)=\frac{1}{2\pi}\lim_{R\to\infty}\int_{-R}^R e^{ixs}\int_{x_1}^{x_2} e^{-ist}f(t)\,dt\,ds.
$$
\end{corollary}
Notice that it is immaterial if $f$ is defined outside the interval $[x_1,x_2]$.
\begin{proof}
With $g=f'-iwf$ we have 
\begin{equation}
f(x)=e^{iwx}\int_{x_1}^x e^{-iwt}g(t)\,dt +f(x_1)e^{iw(x-x_1)}.
\end{equation}
The proof is similar to the corresponding result in Theorem~\ref{theoremtransform}.  Using
Lemma~\ref{lemmaperron1} we have
$$
f(x)  =  \frac{1}{2\pi i}\int_{x_1}^{x_2}\lim_{R\to\infty}\int_{-R}^R \frac{e^{i(x-t)z}}{z-w}\,dz\, g(t)\,dt
+f(x_1)e^{iw(x-x_1)}.
$$
The Heaviside function $H(x-t)$ lets us replace the upper limit of integration in $t$ with $x_2$.
As before, we can bring the $R$ limit outside the integrals and interchange the order of integration.
Then integration by parts yields
\begin{eqnarray*}
f(x) & = & \frac{1}{2\pi i}\lim_{R\to\infty}\int_{-R}^R\frac{e^{ixz}}{z-w}
\left\{\left[e^{-itz}f(t)\right]_{x_1}^{x_2}+i(z-w)\int_{x_1}^{x_2}e^{-itz}f(t)\,dt\right\}dz\\
 & & \qquad+f(x_1)e^{iw(x-x_1)}\\
 & = & f(x_2)H(x-x_2)e^{iw(x-x_2)}+\frac{1}{2\pi}\lim_{R\to\infty}\int_{-R}^Re^{ixz}\int_{x_1}^{x_2}e^{-itz}f(t)\,dt\,dz\\
 & = & \frac{1}{2\pi}\lim_{R\to\infty}\int_{-R}^Re^{ixz}\int_{x_1}^{x_2}e^{-itz}f(t)\,dt\,dz.
\end{eqnarray*}
\end{proof}

To see localization in practice consider the following example.
\begin{example}
First take $f(t)=1$. Then
$$
\int_{x_1}^{x_2}f(t)e^{-ist}\,dt=i\left(\frac{e^{-is{x_2}}-e^{-isx_1}}{s}\right).
$$
And,
\begin{eqnarray*}
i\int_{-R}^R e^{ixs}\left(\frac{e^{-is{x_2}}-e^{-isx_1}}{s}\right)ds 
 & = & \int_{-R}^R \frac{\sin((x_2-x)s)+\sin((x-x_1)s)}{s}ds\\
 & = & \int_{-R(x_2-x)}^{R(x_2-x)}\frac{\sin(s)}{s}ds +\int_{-R(x-x_1)}^{R(x-x_1)}\frac{\sin(s)}{s}ds.
\end{eqnarray*}
As $R\to\infty$ this approaches $2\intinf \frac{\sin(s)}{s}ds=2\pi$, in accordance with 
Theorem~\ref{corollarylocalisation}.  The integral $\intinf \frac{\sin(s)}{s}ds=\pi$ is
computed in Lemma~\ref{lemmaperron1}.

Now consider $f(t)=t^n$ for $n\in\N$.  Then
$$
\int_{x_1}^{x_2}f(t)e^{-ist}\,dt  =  i^n\frac{\partial^n}{\partial s^n}\int_{x_1}^{x_2}e^{-ist}\,dt
  =  i^{n+1}\frac{\partial^n}{\partial s^n}\left(\frac{e^{-is{x_2}}-e^{-isx_1}}{s}\right).
$$
And, integrating by parts $n$ times,
\begin{align}
&i^{n+1}\int_{-R}^R e^{ixs}\frac{\partial^n}{\partial s^n}\left(\frac{e^{-is{x_2}}-e^{-isx_1}}{s}\right)ds\\
&\quad=i^{n+1}\Bigg\{\sum_{k=1}^n(-ix)^{k-1} \left[e^{ixs}\frac{\partial^{n-k}}{\partial s^{n-k}}\left(\frac{e^{-is{x_2}}-e^{-isx_1}}{s}\right)
\right]_{s=-R}^R\label{exampleR}\\
&\qquad\qquad+(-1)^n(ix)^n\int_{-R}^R e^{ixs}\left(\frac{e^{-is{x_2}}-e^{-isx_1}}{s}\right)ds\Bigg\}.\label{exampleRR}
\end{align}
As $R\to\infty$, the term in \eqref{exampleR} vanishes and, as in the case $n=0$ above, the term in 
\eqref{exampleRR} tends to
$2\pi x^n$.
\end{example}

\section{Multivariable}
The method used in Section~\ref{sectionft} to obtain the Fourier integral
representation applies with little change to functions of
several variables.  We will show this for functions of two variables from which it will
be clear how to proceed with $n$ variables.

For functions of two or more variables there is more than one type of absolute continuity.
The most appropriate for the
problem at hand is the Carath\'{e}odory definition, which we discuss below.
\begin{theorem}\label{theoremtransformxy}
Let $f\in C^2(\R^2)$ such that the functions $f$, $f_x$, $f_y$ and $f_{xy}$ are all in $L^1(\R^2)$.
Then for each $(x,y)\in\R^2$,
$$
f(x,y)=\frac{1}{(2\pi)^2}\lim_{R_1\to\infty}\lim_{R_2\to\infty}\int_{-R_1}^{R_1}\int_{-R_2}^{R_2}\! e^{i(xs+yt)}
\intinf\intinf\! e^{-i(su+tv)}f(u,v)\,dv\,du\,dt\,ds.
$$
\end{theorem}
We have given elementary conditions under which the inversion theorem holds.  Following the proof of the
theorem we consider how these can be weakened.

\begin{proof}
Define function $g\fn\R^2\to\R$ by $g(x,y)=f_{xy}(x,y)-iw[f_x(x,y)+f_y(x,y)]-w^2f(x,y)$ where $w=\xi+i\eta$ with $\xi,\eta\in\R$,
$\eta>0$.  Subscripts denote partial derivatives with respect to the indicated variable.  Then
\begin{equation}
f(x,y)=e^{iw(x+y)}\int_{-\infty}^x\int_{-\infty}^y e^{-iw(s+t)}g(s,t)\,dt\,ds.\label{fxy}
\end{equation}
Partial differentiation and the Fubini--Tonelli theorem show that $f$ satisfies the partial differential equation.
To show all solutions
of this partial differential equation are given by this formula, suppose there were two solutions and
$F$ was their difference.  Then
\begin{eqnarray}
0 & = & F_{xy}(x,y)-iw[F_x(x,y)+F_y(x,y)]-w^2F(x,y)\notag\\
 & = & e^{iw(x+y)}\partial_{xy}\left[e^{-iw(x+y)}F(x,y)\right].\label{partiale}
\end{eqnarray}
The solutions of the partial differential equation $\phi_{xy}=0$ are $\phi(x,y)=A(x)+B(y)$
for differentiable functions $A$ and $B$ of one variable.  But then
$F(x,y)=e^{i(\xi+i\eta)(x+y)}(A(x)+B(y))$.   By the Fubini--Tonelli theorem the function
$x\mapsto F(x,y)$ is in $L^1(\R)$ for almost all $y\in\R$.  So the function
$x\mapsto e^{-\eta x}\abs{A(x)+B(y)}$ is in $L^1(\R)$ for almost all $y\in\R$.
And the function
$y\mapsto e^{-\eta y}\abs{A(x)+B(y)}$ is in $L^1(\R)$ for almost all $x\in\R$.  This shows $A$ and
$B$ are constants with $A+B=0$.

Using Lemma~\ref{lemmaperron1} we can write
\begin{eqnarray}
f(x,y) & = & \int_{-\infty}^{\infty}\int_{-\infty}^{\infty}H(x-s)H(y-t)
e^{iw(x-s)}e^{iw(y-t)}g(s,t)\,dt\,ds\label{fHH}\\
 & = & -\frac{1}{(2\pi)^2}\int_{-\infty}^{\infty}\int_{-\infty}^{\infty}
\int_{-\infty}^{\infty}\frac{e^{i(x-s)z_1}}{z_1-w}\int_{-\infty}^{\infty}\frac{e^{i(y-t)z_2}}{z_2-w}
g(s,t)\,dz_2\,dz_1\,dt\,ds.\notag
\end{eqnarray}
The rest of the proof follows as in the proof of Theorem~\ref{theoremtransform}.
We need the following limits:
$\lim_{\abs{y}\to\infty}f(x,y)=0$ for each $x$,
$\lim_{\abs{x}\to\infty}f(x,y)=0$ for each $y$,
$\lim_{\abs{y}\to\infty}f_x(x,y)=0$ for each $x$ (or, if the $s$ and $t$ integrals are interchanged,
$\lim_{\abs{x}\to\infty}f_y(x,y)=0$ for each $y$).
These follow from the hypotheses of the theorem and Lemma~\ref{lemmaff'}.
\end{proof}

The hypothesis $f\in C^2(\R^2)$ in the theorem can be weakened in several ways.  To prove $f$ 
satisfies the partial differential equation we only need $f_{xy}=f_{yx}$
almost everywhere.  There is an extensive literature on equality of mixed partial
derivatives.  See \cite{minguzzi} and references.  See \cite{mykhaylyuk} for conditions for
equality and seemingly reasonable conditions under which the mixed partial derivatives can
fail to be equal on a set of positive measure. Some of these
conditions are quite complicated.  A simple one due to Currier \cite{currier} is existence
of $f_x$ and $f_y$ everywhere and existence of $f_{xx}$,  $f_{xy}$,  $f_{yx}$,  $f_{yy}$
almost everywhere.

In \eqref{fxy}, the two iterated integrals are equal by
the Fubini--Tonelli theorem, provided $f$, $f_x$, $f_y$ and $f_{xy}$ are all in $L^1(\R^2)$.
The usual version is proved, for example, in
\cite{folland}.  An extension is given by Moricz in \cite{moricz}.

Following \eqref{fHH} we integrate by parts.
This requires that for almost all $x\in\R$ the function
$y\mapsto f_x(x,y)$ be absolutely continuous on $\R$ and $\lim_{\abs{y}\to\infty}f_x(x,y)=0$.
Notice that the assumptions $f_x\in L^1(\R^2)$ and $f_{xy}\in L^1(\R^2)$ and the
Fubini--Tonelli theorem tell us that the functions
$y\mapsto f_x(x,y)$ and $y\mapsto f_{xy}$ are in $L^1(\R)$ for almost all $x\in\R$.
Lemma~\ref{lemmaff'} then says $\lim_{\abs{y}\to\infty}f_x(x,y)=0$.
Or, if the $s$ and $t$ integrals in \eqref{fHH} are interchanged we can  repeat the above with $x$ and $y$ interchanged.  Then we also require
the function $x\mapsto f(x,y)$ be absolutely continuous on $\R$ for almost all $y\in\R$ and
$\lim_{\abs{x}\to\infty}f(x,y)=0$.  As above, this limit follows from Lemma~\ref{lemmaff'}
and the assumptions $f,f_x\in L^1(\R^2)$.
And a similar condition with $x$ and $y$ interchanged.

\begin{corollary}\label{corollarycaratheodory}
The conclusion of Theorem~\ref{theoremtransformxy} holds if
$f\fn\R^2\to\R$ such that $f_x$ and $f_y$ exist on $\R^2$ and
$f_{xx}$,  $f_{xy}$,  $f_{yx}$,  $f_{yy}$
exist almost everywhere such that
$f$, $f_x$, $f_y$ and $f_{xy}$ are all in $L^1(\R^2)$;
the function $x\mapsto f(x,y)$ is absolutely continuous on $\R$ for each $y\in\R$,
the function $y\mapsto f(x,y)$ is absolutely continuous on $\R$ for each $x\in\R$,
the function $y\mapsto f_x(x,y)$ is absolutely continuous on $\R$ for almost every $x\in\R$
(or, the function $x\mapsto f_y(x,y)$ is absolutely continuous on $\R$ for almost every $y\in\R$).
\end{corollary}

For $L^1$ functions of one variable, absolute continuity gives 
necessary and sufficient conditions under which derivatives can
be integrated and under which indefinite integrals can be differentiated.
There does not appear to be a single such condition for integration in
$\R^2$.  There are notions of absolute continuity for functions of 
two variables due to Carath\'{e}odory, Tonelli, and other authors.
Each extends some properties of absolute continuity on the real line to the
plane, while other properties are lost.   See \cite[p.~169]{saks} for Tonelli's definition
and
\cite{sremr} for references to other authors.  

Carath\'{e}odory's definition seems the most relevant here.
\begin{defn}[Carath\'{e}odory absolute continuity]
Let $F\fn\R^2\to\R$ such that
\begin{enumerate}
\item[(a)]
For each $\epsilon>0$ there is $\delta>0$ such that if $P_i$ are mutually disjoint open
intervals with $\sum_{j=1}^k\abs{P_j}<\delta$ then $\sum_{j=1}^k\abs{F(P_j)}<\epsilon$.
Here, $\abs{P_j}$ is the area of rectangle $P_j$ and $F((x_1,y_1)\times(x_2,y_2))
=F(x_1,y_1)-F(x_2,y_1)-F(x_1,y_2)+F(x_2,y_2)$.
\item[(b)]
The function $x\mapsto F(x,y_0)$ is absolutely continuous on the real line for some fixed $y_0\in\R$;
the function $y\mapsto F(x_0,y)$ is absolutely continuous on the real line for some fixed $x_0\in\R$.
\end{enumerate}
\end{defn}

This is a slight modification of Carath\'{e}odory's definition. See \cite{sremr} where there is
also a 
reference to Carath\'{e}odory's original work.   This type of absolute continuity implies equality of the
mixed partial derivatives and the one-variable absolute continuity conditions in Corollary~\ref{corollarycaratheodory}.
It also implies existence of a function $\phi\in L^1(\R^2)$ so that
$f(x,y)=\int_{-\infty}^x\int_{-\infty}^y \phi(s,t)\,dt\,ds +\chi(x)+\psi(y)$ for some absolutely continuous
functions $\chi$, $\psi$ on the real line, as we saw in the proof of Theorem~\ref{theoremtransformxy}.
Hence, the hypotheses can be replaced with Carath\'{e}odory absolute continuity.
\begin{corollary}\label{corollarycaratheodory2}
The conclusion of Theorem~\ref{theoremtransformxy} holds if
$f\fn\R^2\to\R$ such that
$f$, $f_x$, $f_y$ and $f_{xy}$ are all in $L^1(\R^2)$ and $f$ is absolutely continuous
in the sense of Carath\'{e}odory.
\end{corollary}

\section{Perron's Lemma}\label{sectionperron}
In this section we prove Lemma~\ref{lemmaperron1}.
\begin{proof}
Write $z=x+iy$ and $w=\xi+i\eta$ with $\eta>0$.  Let $R$ be a real
number larger than $2\abs{w}$. 

Suppose $p>0$.  Let $\Gamma_R$ be the upper half-circle $z=Re^{i\theta}$ for $0\leq\theta\leq\pi$.   
Let $\Gamma$ be the closed contour consisting of $\Gamma_R$ oriented counterclockwise and
the $x$-axis from $x=-R$ to $x=R$.
The integrand is analytic in $\Gamma$ except for a simple pole at $z=w$, with residue $e^{ipw}$.
By Cauchy's theorem,
\begin{equation}
\frac{1}{2\pi i}\oint_\Gamma\frac{e^{ipz}}{z-w}dz= e^{ipw}.\label{cauchy}
\end{equation}

Now show the integral over $\Gamma_R$ vanishes as $R$ tends to
infinity.  
On $\Gamma_R$ we have
$$
I_R=\int_{0}^{\pi}\frac{e^{ipR(\cos\theta+i\sin\theta)}iRe^{i\theta}\,d\theta}{Re^{i\theta}-w}
$$
so that 
$$
\abs{I_R}  \leq  \frac{2R}{R-\abs{w}}\int_{0}^{\pi/2}e^{-pR\sin\theta}\,d\theta.
$$
For $0\leq\theta\leq\pi/2$ we have $\sin\theta\geq 2\theta/\pi$ (Jordan's inequality).  This gives
\begin{equation}
\abs{I_R}  \leq  \frac{2}{1-\abs{w}/R}\int_{0}^{\pi/2}e^{-2pR\theta/\pi}\,d\theta
  \leq  \frac{2\pi}{pR}\left[1-e^{-pR}\right]
  \to  0 \text{ as } R\to\infty.\label{contourestimate}
\end{equation}

When $p<0$ we use the lower half-circle given by $z=Re^{i\theta}$ for $-\pi\leq\theta\leq 0$.  The
analysis is similar except that now the integrand is analytic and we get zero for the required
integral.

When $p=0$ the integral exists only in the principal value sense.  We have
\begin{eqnarray*}
\int_{-R}^R \frac{dx}{(x-\xi)-i\eta}   & = &  \int_{-R}^R \frac{(x-\xi)\,dx}{(x-\xi)^2+\eta^2}
+ i\eta\int_{-R}^R \frac{dx}{(x-\xi)^2+\eta^2}\\
 & = & \frac{1}{2}\log\left[\frac{(1-\xi/R)^2+(\eta/R)^2}{(1+\xi/R)^2+(\eta/R)^2}\right]\\
 & & \qquad+i\left[
\arctan\left(\frac{R-\xi}{\eta}\right)+\arctan\left(\frac{R+\xi}{\eta}\right)\right]\\
 & \to &  i\pi \text{ as } R\to\infty.
\end{eqnarray*}

If $w=0$ and $p>0$ then 
$$
\int_{-R}^R \frac{e^{ipz}}{z}dz=2i\int_{0}^R \frac{\sin(px)}{x}dx
$$
since the real part is odd.  The same calculation as above shows $I_R\to 0$ as $R\to\infty$.
Now the pole is at the origin so we need to let $z=\epsilon e^{i\theta}$ and integrate over
a semicircle of radius $\epsilon$ for $0\leq\theta\leq\pi$.  Using dominated convergence, this
becomes
$$
\int_0^\pi\frac{e^{ip\epsilon(\cos\theta+i\sin\theta)}\epsilon i e^{i\theta}}{\epsilon
e^{i\theta}}d\theta=i \int_0^\pi e^{ip\epsilon\cos\theta}e^{-p\epsilon\sin\theta}\,d\theta\to i\pi
 \text{ as } \epsilon\to 0.
$$
\end{proof}

When $p\not=0$ and the imaginary part of $w$ is positive 
it is also possible to use a rectangular contour with one edge on the
$x$-axis.
This shows that the two integrals $\int_0^\infty [e^{ipz}/(z-w)]\,dz$
and $\int^0_{-\infty}[e^{ipz}/(z-w)]\,dz$ exist independently.
However, the estimate on $I_R$ used in the proof of Theorem~\ref{theoremtransform} required
the symmetric form.

\end{document}